\documentclass[11pt]{amsart}
\usepackage{amsthm}
\usepackage{array}
\usepackage{amsmath}
\usepackage{enumerate}
\usepackage{tikz}
\usetikzlibrary{calc}
\usepackage{amssymb}
\usepackage{latexsym}
\usepackage{amsfonts}
\usepackage{color}
\usepackage{mathrsfs}
\usepackage{epsfig,helvet}

\theoremstyle{plain} \numberwithin{equation}{section}
\newtheorem{thm}{Theorem}[section]
\newtheorem{theorem}[thm]{Theorem}
\newtheorem{lemma}[thm]{Lemma}
\newtheorem{corollary}[thm]{Corollary}
\newtheorem{example}[thm]{Example}
\newtheorem{definition}[thm]{Definition}

\begin{document}
\setcounter{page}{1}

\title[Some properties of factor set in regular Hom-Lie algebras]{Some properties of factor set in regular Hom-Lie algebras}

\author[Padhan]{Rudra Narayan Padhan}
\address{Department of Mathematics, National Institute of Technology  \\
         Rourkela,
          Odisha-769028 \\
                India}
\email{rudra.padhan6@gmail.com}
\author[Nandi]{Nupur Nandi}
\address{Department of Mathematics, National Institute of Technology,  \\
         Rourkela, 
          Odisha-769028 \\
           India}
\email{nupurnandi999@gmail.com}
\author[Pati]{K. C Pati}
\address{Department of Mathematics, National Institute of Technology  \\
         Rourkela, 
          Odisha-769028 \\
                India}
\email{kcpati@nitrkl.ac.in}

\subjclass[2010]{Primary 17B30; Secondary 17B05.}
\keywords{ Hom-Lie algebra; Isoclinism; Factor set }
\maketitle
\begin{abstract}
 In this paper, we give the definition of isoclinism for regular Hom-Lie algebras and verify some of its properties. Finally, we introduce the factor set and show that the isoclinism and isomorphism of two finite same dimensional regular Hom-Lie algebras are equivalent.
\end{abstract}

\section{Introduction}
Hom-Lie algebras have been investigated widely for the last few years. Actually, the Hom-Lie algebras were first studied by Hartwig, Larsson and Silvestrov in [4] as part of a study of deformations of the Witt and the Virasoro algebras. There were some successive contributions on Hom-Lie algebras in [1,2,8]. A Hom-Lie algebra is a triple $(V,[.,.],\varphi)$ in which the bracket satisfies a twisted Jacobi identity along the linear map $\varphi$. It should be pointed out that Lie algebras form a particular case of Hom-Lie algebras, i.e. when $\varphi$ equals to the identity map.
\par
 Philip Hall introduced isoclinism for groups in 1940. Then in 1994 Kay Moneyhun derived a its Lie algebra version [5,6]. The concept of factor set plays a dominant role in the studies of isoclinism, e.g. in finite dimensional case, isoclinism and isomorphism are equivalent. Thus we wonder whether the same is true for Hom-Lie algebras or not, but eventually the answer is yes for regular Hom-Lie algebras. Although the ideas followed the work [5,7,9], several important results pertaining particularly to the case of regular Hom-Lie algebras are proved. In future, we hope these results would be helpful to study the stem cover and Schur multiplier of Hom-Lie algebras. Already some works have been done for central extension and Cohomology on Hom-Lie algebras [11,12]. Thus it will be of most interest to extend Lie algebras results to that of regular Hom-Lie algebras.
\par
Throughout this paper all algebras and vector spaces are considered over $\mathbb{F}$, a field of characteristic $0$.

\begin{definition}
	A Hom-Lie algebra is a triple $(V,[.,.],\varphi)$ consisting of a vector space $V$, a skew-bilinear map $[.,.]:V\times V\rightarrow V$ and a linear map $\varphi:V\rightarrow V$ that satisfies Hom-Jacobi~ identity i.e.
	
	\[ [\varphi(v_1),[v_2,v_3]]+[\varphi(v_2),[v_3,v_1]]+[\varphi(v_3),[v_1,v_2]]=0,~\forall~ v_1,v_2,v_3\in V \]
\end{definition}
\noindent Let $(V,[.,.],\varphi)$ be a Hom-Lie algebra. A Hom-Lie subalgebra of $(V,[.,.],\varphi)$ is a vector subspace $W$ of $V$, which is closed by the bracket and is invarient by $\varphi$, i.e.
\begin{enumerate}
	\item $[w_1,w_2]\in W~ \forall~w_1,w_2\in W$
	\item $\varphi(w) \subseteq W~\forall ~w\in W$
\end{enumerate}
A Hom-Lie subalgebra of $(V,[.,.],\varphi)$ is said to be a ideal if $[w,v]\in W$ for all $w\in W,v\in V$.\\
The centre of a Hom-Lie algebra $(V,[.,.],\varphi)$ is defined by $$Z(V)=\{x\in V: [x,v]=0~\forall~ v\in V\}.$$
 An abelian Hom-Lie algebra is a vector space $V$ endowed with trivial
bracket and any linear map $\varphi: V\rightarrow V$.

A Hom-Lie algebra $(V,[.,.],\varphi)$ is said to be $multiplicative$ if  $\varphi([v_1,v_2])=[\varphi(v_1),\varphi(v_2)]$ for all $v_1,v_2\in V$. A multiplicative Hom-Lie algebra $(V,[.,.],\varphi)$ is said to be $regular$ if $\varphi$ is bijective. It should be noted that $Z(V)$ is not always an ideal of $(V,[.,.],\varphi)$.\\

\begin{lemma}
If $(V,[.,.],\varphi)$ is regular Hom-Lie algebra, then $Z(V)$ is an ideal of $V$.
\end{lemma}
\begin{proof}
Let $x \in Z(V)$, then for any $v \in V$ there exists $v' \in V$ such that $\varphi (v')=v$. Now \[ [\varphi(x),v]= [\varphi(x),\varphi(v')]=\varphi[x,v']=0, \]
which implies $\varphi(x)\in Z(V)$. Also for $x,y \in Z(V)$, then 
\[[[x,y],v]=[[x,y],\varphi(v')]=-[[y,v'],\varphi(x)]-[[v',x],\varphi(y)]=0.\]
Thus, $[x,y] \in Z(V)$ and $Z(V)$ is an ideal.

\end{proof}
The derived subalgebra of a Hom-Lie algebra $(V,[.,.],\varphi)$ is the subspace of $V$ generated by the elements of the form $[v_1,v_2]$ with $v_1,v_2 \in V$ and we denoted it by $V'$.
 A regular Hom-Lie algebra $(V,[.,.],\varphi)$ is called stem Hom-Lie algebra whenever $Z(V)\subseteq V^\prime$.

\begin{definition}
	Let  $(V,[.,.]_1,\varphi _1)$ and  $(W,[.,.]_2,\varphi_2)$ be two Hom-Lie algebras. A linear map $f: V\rightarrow W$ is a Hom-Lie algebra morphism if for all $v_1,v_2\in V$, $f([v_1,v_2]_{1})=[f(v_1),f(v_2)]_{2}$ and $ f \circ \varphi_1=\varphi_2\circ f$. In other words, the following diagram commutes:
	\begin{center}
		\begin{tikzpicture}[>=latex]
		\node (x) at (0,0) {\(V\)};
		\node (z) at (0,-2) {\(V\)};
		\node (y) at (2,0) {\(W\)};
		\node (w) at (2,-2) {\(W\)};
		\draw[->] (x) -- (y) node[midway,above] {$f$};
		\draw[->] (x) -- (z) node[midway,left] {$\varphi_1$};
		\draw[->] (z) -- (w) node[midway,below] {$f$};
		\draw[->] (y) -- (w) node[midway,right] {$\varphi_2$};
		\end{tikzpicture}\\
	\end{center}
	In particular, they are isomorphic if $ f$ is a bijective linear map.
\end{definition}

\begin{lemma}
Let $f:(V,[.,.]_1,\varphi _1)\longrightarrow (W,[.,.]_2,\varphi_2)$ be an isomorphism. If $(V,[.,.]_1,\varphi _1)$ is regular, then $(W,[.,.]_2,\varphi _2)$ is also.

\end{lemma}
\begin{proof}
 Let $w_1,w_2 \in W$, then there exists $v_1,v_2 \in V$ such that $f(v_1)=w_1, f(v_2)=w_2$.   Together with $f \varphi_1=\varphi_2 f$ and $\varphi_1$ is regular , we have
\[\varphi_{2}[w_1,w_2]_2=\varphi_{2}([f(v_1),f(v_2)]_2)=\varphi_{2} f([v_1,v_2]_1)=f \varphi_{1}([v_1,v_2]_1)=f([\varphi_1(v_1),\varphi_1(v_2)]_1)=[f\varphi_1(v_1),f\varphi_1(v_2)]_2\]
\[=[\varphi_2 f (v_1), \varphi_2 f (v_2)]_2=[\varphi_2(w_1),\varphi_2(w_2)]_2.\]

\noindent Now $f$ and $\varphi$ are bijective, therefore $\varphi_2$ . Thus, $(W,[.,.],\varphi _2)$ is regular. 
\end{proof}
For any ideal $I$ of $(V,[.,.],\varphi )$, we can naturally define the quotient Hom-Lie algebra on the quotient vector space $V/I$ by defining; $[.,.]:V/I \times V/I \longrightarrow V/I$ by \[[\overline{v_1},\overline{v_2}]=\overline{[v_1,v_2]},~\forall~\overline{v_1},\overline{v_2}\in V/I  \]
and $\tilde{\varphi}:V/I\rightarrow V/I$ is naturally induced by $\varphi$ i.e.
\[\tilde{\varphi}(\overline{v})=\varphi(v)+I.\]
Then with this skew-bilinear and linear map, $(V/I,[.,.],\tilde{\varphi})$ is a Hom-Lie algebra and we called it the quotient Hom-Lie algebra.\\

Let  $(V,[.,.]_1,\varphi _1)$ and  $(W,[.,.]_2,\varphi_2)$ be two Hom-Lie algebras. Then we can define the direct sum of these Hom-Lie algebras by defining;
\[[(v_1,w_1),(v_2,w_2)]:=([v_1,v_2],[w_1,w_2]) \]
and $\psi : V\oplus W \longrightarrow V\oplus W$ by
\[\psi (v,w)=(\varphi_{1}(v),\varphi_{2}(w)).\]
Then $(V\oplus W,[.,.],\psi)$ is a Hom-Lie algebra.\\

\noindent  In section-2, the definition of isoclinism is given and we prove some of its properties for regular Hom-Lie algebra. In section-3, we introduce factor set in regular Hom-Lie algebra. Moreover, we show the existence of factor set for any regular Hom-Lie algebra and prove that the two finite dimensional regular Hom-Lie algebras are isoclinic if and only if they are isomorphic.

\section{Isoclinism}

 We can only define the isoclinism for regular Hom-Lie algebra as isoclinism can not be defined for an arbitrary Hom-Lie algebra. We shall like to mention for Hom-Lie algebra which is not a regular one, $Z(V)$ will not be an ideal of $V$ and we can not define a quotient algebra $V/Z(V)$ which is required for the definition of isoclinism.

\begin{definition}
Let $(V,[.,.]_1,\varphi_1)$ and $(W,[.,.]_2,\varphi_2)$ be two regular Hom-Lie algebras, $\alpha: \frac{V}{Z(V)}\longrightarrow \frac{W}{Z(W)}$ and $\beta :V' \longrightarrow W'$ be Hom-Lie algebra homomorphisms such that the following diagram is commutative,
\begin{center}
\begin{tikzpicture}[>=latex]
\node (x) at (0,0) {\(\frac{V}{Z(V)}\times \frac{V}{Z(V)} \)};
\node (z) at (0,-2) {\(\frac{W}{Z(W)}\times \frac{W}{Z(W)}\)};
\node (y) at (3,0) {\(V'\)};
\node (w) at (3,-2) {\(W'\)};
\draw[->] (x) -- (y) node[midway,above] {$\mu$};
\draw[->] (x) -- (z) node[midway,left] {$\alpha ^{2} $};
\draw[->] (z) -- (w) node[midway,below] {$\rho$};
\draw[->] (y) -- (w) node[midway,right] {$\beta$};
\end{tikzpicture}\\
 \end{center} 
where $\mu( (\overline{v_{1}},\overline{v_{2}})):=[\overline{v_{1}},\overline{v_{2}}]_1$ for $ v_1,v_2 \in V $ and similarly  $\rho( (\overline{w_{1}},\overline{w_{2}})):= [\overline{w_{1}},\overline{w_{2}}]_2$ where $w_{1},w_ {2} \in W$. Then the pair $(\alpha, \beta)$ is called {\it homoclinism} and if they are both isomorphisms, then $(\alpha, \beta)$ is called {\it isoclinism}.
\end{definition}
 If $(\alpha, \beta)$ is an isoclinism between $(V,[.,.]_1,\varphi_1)$ and $(W,[.,.]_2,\varphi_2)$, then $(V,[.,.]_1,\varphi_1)$ and $(W,[.,.]_2,\varphi_2)$ are said to be isoclinic, which is denoted by $V\sim W$. We observe that isoclinism is an equivalence relation. 
\begin{lemma}
	If $(V,[.,.]_1,\varphi_1)$ is a regular Hom-Lie algebra and $(W,[.,.]_2,\varphi_2)$ be an abelian Hom-Lie algebra, then $V\sim V\oplus W$.
\end{lemma}

\begin{proof}
	
	Let $(V\oplus W,\psi)$ be the Hom-Lie algebra. As $W$ is abelian, $Z(V\oplus W)=Z(V)\oplus W$. Then $(\frac{V\oplus W}{Z(V)\oplus W},\tilde{\psi})$ and $(V/Z(V),\tilde{\varphi_1})$ are the quotient Hom-Lie algebra. Define the map $$\alpha :\frac{V}{Z(V)}\rightarrow \frac{V\oplus W}{Z(V)\oplus W}$$ by $$v+Z(V)\rightarrow (v,0)+(Z(V)\oplus W) $$ for all $v\in V$. It is easy to verify that $\alpha$ is well defined bijection map and for $v_1,v_2\in V$, $\alpha([v_1+Z(V),v_2+Z(V)])=[\alpha(v_1+Z(V)),\alpha(v_2+Z(V))]$ holds, implies that $\alpha$ is a homomorphism. Also we will check now that $\alpha \tilde{\varphi_1}= \tilde{\psi} \alpha$, i.e. 
	\[ \alpha \tilde{\varphi_1}(\overline{v})=\alpha (\varphi (v)+Z(V))=(\varphi(v),0)+ (Z(V)\oplus W)= \tilde{\psi}((v,0)+(Z(V)\oplus W)=\tilde{\psi} \alpha(\overline{v}).\]

	Also if we consider the identity map $\beta : V^\prime \rightarrow (V\oplus W)^\prime = V^\prime$, then we get the desired results .
	
\end{proof}

 We list some results whose proof are all most similar, one can find the proofs in the following papers [5,9].

\begin{lemma}
	Let $(V,[.,.],\varphi)$ be a Hom-Lie algebra and $I$ be an ideal. Then $V/I \sim V/{(I\cap V^\prime)}$. In particular, if $I\cap V^\prime=0$ then $V\sim V/I$. Conversely if $V^\prime$  is finite dimensional and $V\sim V/I$ then $I\cap V^\prime=0$.
\end{lemma}

\begin{corollary}
	Let $(V,[.,.]_1,\varphi_1)$ and $(W,[.,.]_2,\varphi_2)$ be two Hom-Lie algebras. If $f:V\rightarrow W$ is an onto homomorphism such that $Ker(f)\cap V^\prime=0$, then $f$ induces a isoclinism between $V$ and $W$.  
\end{corollary}

\begin{lemma}
	Suppose $\mathcal{C}$ is a isoclinic family of regular Hom-Lie algebas. Then
	\\(1) $\mathcal{C}$ contains a stem Hom-Lie algebra.
	\\(2) Each finite dimensional regular Hom-Lie algebra $T\in \mathcal{C}$ is stem if and only if $T$ has minimal dimension in $\mathcal{C}$.
\end{lemma}

\begin{lemma}
	Let the pair $(\alpha , \beta)$ be isoclinism of regular Hom-Lie algebras $(V,[.,.]_1,\varphi_{1})$ and $(W,[.,.]_2,\varphi_{2})$. Then the following statements hold: \\
	(1) $\alpha(v+Z(V))=\beta(v)+Z(W)$;\\
	(2) $\beta([v_1,v_2])=[\beta(v_1),v_3]$, for all $v_1\in V^\prime,v_2\in V$ and $v_3+Z(W)=\alpha(v_2+Z(L))$.
\end{lemma}

\begin{lemma}
Let $(V,[.,.],\varphi)$ be a regular Hom-Lie algebra and $V= U \oplus Z(V)$, then $\varphi(U)\subseteq U$.
\end{lemma}
\begin{proof}
Suppose $\varphi(u) \in Z(V)$ for some $0\neq u \in U$. Then for every $y \in V$ there exist $x \in V$ such that $\varphi(x)=y$. Thus 
\[0=[\varphi(u),y]=[\varphi(u),\varphi(x)]=\varphi[u,x].\] 
Since $\varphi$ is injective, $[u,x]=0$ for all $x \in V$. Thus $u \in Z(V)$, which is not true. Therefore $\varphi(u) \in U$.
\end{proof}

\section{Factor set in hom-Lie algebras}\label{sec2}
 From onwards we denote a Hom-Lie algebra by $(V,\varphi)$. We have seen in Lemma 1.2 that $Z(V)$ is an ideal if $(V,\varphi)$ is regular, because of this result factor set can be studied only for regular Hom-Lie algebras. 
\begin{definition}
Let $(V,\varphi)$ be a finite dimensional regular Hom-Lie algebra. A skew-bilinear map;
              \[r:V/Z(V)\times V/Z(V) \longrightarrow Z(V)\]
is said to be a $factor~set$ if for all $\overline{v_1}, ~ \overline{v_2},  ~\overline{v_3} \in V/Z(V)$,

 $$r([\overline{v_1},\overline{v_2}],\tilde{\varphi}(\overline{v_3}))+r([\overline{v_2},\overline{v_3}],\tilde{\varphi}(\overline{v_1}))+r([\overline{v_3},\overline{v_1}],\tilde{\varphi}(\overline{v_2}))=0.$$
 
\noindent  Where $\tilde{\varphi}:V/Z(V)\rightarrow V/Z(V)$ defined by $\tilde{\varphi}(\overline{v}):=\varphi(v)+Z(V).$
The factor set $r$ is said to be multiplicative if \[r(\tilde{\varphi}(\overline{v_1}),\tilde{\varphi}(\overline{v_2}))=\varphi r(\overline{v_1},\overline{v_2}) ~~~~\forall ~\overline{v_1},\overline{v_2} \in V/Z(V). \]
\end{definition}
The next Lemma tells us that we can always construct a new Hom-Lie algebra from a given regular Hom-Lie algebra and a factor set on it. The below Lemma is the one of the important main tool of this paper.

\begin{lemma} \label{Lem5}
Let $(V,\varphi)$ be a regular Hom-Lie algebra and $r$ be a factor set on $(V,\varphi)$, we define a set
\[R=(Z(V),V/Z(V),r)=\{(x,\overline{v}):~x \in Z(V), \overline{v} \in V/Z(V)\}. \]
Then
\begin{enumerate}
\item $(R,\psi)$ is a Hom-Lie algebra under the component-wise addition and the skew-bilinear map defined by
\[[(x_{1},\overline{v_1}),(x_{2},\overline{v_2})]:=(r(\overline{v_1},\overline{v_2}),[\overline{v_1},\overline{v_2}])~~~~~~\forall (x_{1},\overline{v_1}),(x_{2},\overline{v_2}) \in R;\]

and the linear map $\psi: R\rightarrow R$ is given by $$\psi ((x,\overline{v})):=(\varphi(x),\tilde{\varphi}(\overline{v}))~~~~~~~~\forall ~(x,\overline{v}) \in R.$$
\item If $r$ is multiplicative, then $(R,\psi)$ is regular.
\item $Z_{R}=\{(x,0) \in R:~x \in Z(V)\} \cong Z(V).$
\end{enumerate}
\end{lemma}
\begin{proof}

  Clearly $[.,.]$ is a well-defined skew-bilinear map. We will verify only the Hom-Jacobi identity. Consider
\begin{align*}
\begin{split}
&[[(x_{1},\overline{v_1}),(x_{2},\overline{v_2})],\psi(x_{3},\overline{v_3})] + [[(x_{2},\overline{v_2}),(x_{3},\overline{v_3})],\psi(x_{1},\overline{v_1})]+[[(x_{3},\overline{v_3}),(x_{1},\overline{v_1})],\psi(x_{2},\overline{v_2})]\\
 & = [(r(\overline{v_1},\overline{v_2}),[\overline{v_1},\overline{v_2}]),(\varphi(x_3),\tilde{\varphi}(\overline{v_3}))]+[(r(\overline{v_2},\overline{v_3}),[\overline{v_2},\overline{v_3}]),(\varphi(x_1),\tilde{\varphi}(\overline{v_1}))]+[(r(\overline{v_3},\overline{v_1}),[\overline{v_3},\overline{v_1}]),(\varphi(x_2),\tilde{\varphi}(\overline{v_2}))]\\
   & = (r([\overline{v_1},\overline{v_2}],\tilde{\varphi}(\overline{v_3})),[[\overline{v_1},\overline{v_2}],\tilde{\varphi}(\overline{v_3})] )+(r([\overline{v_2},\overline{v_3}],\tilde{\varphi}(\overline{v_1})),[[\overline{v_2},\overline{v_3}],\tilde{\varphi}(\overline{v_1})] )+(r([\overline{v_3},\overline{v_1}],\tilde{\varphi}(\overline{v_2})),[[\overline{v_3},\overline{v_1}],\tilde{\varphi}(\overline{v_2})] )\\
    & = \bigg(r([\overline{v_1},\overline{v_2}],\tilde{\varphi}(\overline{v_3}))+r([\overline{v_2},\overline{v_3}],\tilde{\varphi}(\overline{v_1}))+ r([\overline{v_3},\overline{v_1}],\tilde{\varphi}(\overline{v_2})), [[\overline{v_1},\overline{v_2}],\tilde{\varphi}(\overline{v_3})]+[[\overline{v_2},\overline{v_3}],\tilde{\varphi}(\overline{v_1})]+[[\overline{v_3},\overline{v_1}],\tilde{\varphi}(\overline{v_2})]\bigg) \\
     &=(0,\overline{0}). \\
   \end{split}
\end{align*}
Thus $(R,\psi)$ is a Hom-Lie algebra. Now assume that $r$ is multiplicative. Then in one hand \[\psi ([(x_{1},\overline{v_1}),(x_{2},\overline{v_2})])=\psi (r(\overline{v_1},\overline{v_2}),[\overline{v_1},\overline{v_2}])=(\varphi r(\overline{v_1},\overline{v_2}),\tilde{\varphi}[\overline{v_1},\overline{v_2}]). \]
On the other side
\[ [\psi(x_{1},\overline{v_1}),\psi(x_{2},\overline{v_2})]=\psi[(\varphi(x_1),\tilde{\varphi}(\overline{v_1})),(\varphi(x_2),\tilde{\varphi}(\overline{v_2}))]=(r(\tilde{\varphi}(\overline{v_1}),\tilde{\varphi}(\overline{v_2})),[\tilde{\varphi}(\overline{v_1}),\tilde{\varphi}(\overline{v_2})]).\]
 Since $(V/Z(V),\tilde{\varphi})$ is multiplicative, therefore $(R,\psi)$ also.

\end{proof}
 Naturally we can think for the existence of the factor set of a given regular Hom-Lie algebra and the connection between them.
\begin{lemma}\label{Lem6}
 For a regular Hom-Lie algebra $(V,\varphi)$ there exist a factor set $r$ such that $$(V,\varphi) \cong (R,\psi).$$
\end{lemma}
\begin{proof}
Let $V=U \oplus Z(V)$. Let us define $f:V/Z(V)\rightarrow V$ by $f(\overline{v})=f(v+Z(V))=f(u+z+Z(V))=u$, for all $\overline{v} \in V/Z(V)$ and $u \in U$ and $z \in Z(V)$. It is easy to see that $f$ is a well-defined linear map. Also 
for $\overline{v_{1}}=u_{1}+z_{1}, \overline{v_2}=u_{2}+z_{2}$, consider $[\overline{v_1},\overline{v_2}]=[k,k']+Z(V)$. Then
\begin{align*} 
 [f(\overline{v_1}),f(\overline{v_2})]-f[\overline{v_1},\overline{v_2}] +Z(V)&= [u_1,u_2]-f(\overline{[u_1,u_2]})+Z(V) \\
 &=[u_1, u_2]- \overline{f(\overline{[u_1, u_2]})}= 0+Z(V).
\end{align*}
Thus $[f(\overline{v_1}),f(\overline{v_2})]-f[\overline{v_1},\overline{v_2}]  \in Z(V)$.
Now define
$r:V/Z(V)\times V/Z(V) \longrightarrow Z(V)$ which is given by $$r(\overline{v_1},\overline{v_2})=[f(\overline{u_1}),f(\overline{u_2})]-f[\overline{u_1},\overline{u_2}].$$ 
 At first observe that $r$ is a skew-bilinear map. Now to check Hom-Jacobi identity we will first prove $f\tilde{\varphi}(\overline{v})=\varphi f(\overline{v})$ for all $\overline{v} \in V/Z(V)$. By Lemma 2.7, we have 
\[f\tilde{\varphi}(\overline{v})=f\tilde{\varphi}(u+z+Z(V))=f(\varphi(u)+Z(V))=\varphi(u)\]

On the other side 
\[\varphi f(\overline{v})=\varphi (f(u+z+Z(V)))=\varphi(u).\]
Consider
 \begin{equation*} 
\begin{split}
r([\overline{v_1},\overline{v_2}],\tilde{\varphi}(\overline{v_3})) &= [f([\overline{v_1},\overline{v_2}]),f\tilde{\varphi}(\overline{v_3})]-f[[\overline{v_1},\overline{v_2}],\tilde{\varphi}(\overline{v_3})]\\
&= [[f(\overline{v_1}),f(\overline{v_2})]+z,\varphi f(\overline{v_3})]-f[[\overline{v_1},\overline{v_2}],\tilde{\varphi}(\overline{v_3})]\\
&= -[[f(\overline{v_2}),f(\overline{v_3})],\varphi f(\overline{v_1})]-[[f(\overline{v_3}),f(\overline{v_1})],\varphi f(\overline{v_2})]+f([[\overline{v_2},\overline{v_3}],\tilde{\varphi}(\overline{v_1})]+f[[\overline{v_3},\overline{v_1}],\tilde{\varphi}(\overline{v_2})])\\
&= -[[f(\overline{v_2}),f(\overline{v_3})],f\tilde{\varphi}(\overline{v_1})]-[[f(\overline{v_3}),f(\overline{v_1})],f\tilde{\varphi}(\overline{v_2})]+f([[\overline{v_2},\overline{v_3}],\tilde{\varphi}(\overline{v_1})])+f([[\overline{v_3},\overline{v_1}],\tilde{\varphi}(\overline{v_2})])\\
&= -r([\overline{v_2},\overline{v_3}],\tilde{\varphi}(\overline{v_1}))-r([\overline{v_3},\overline{v_1}],\tilde{\varphi}(\overline{v_2})).\\
\end{split}
\end{equation*}
for all $\overline{v_1}, \overline{v_2}, \overline{v_3} \in V/Z(V)$.
Denote $(Z(V),V/Z(V),r)=R$. Define $\theta:(Z(V),V/Z(V),r)\rightarrow V$  by $$\theta(x,\overline{v})= x+f(\overline{v})~~ \forall  ~x \in  Z(V), \overline{v} \in V/Z(V).$$

\noindent  Evidently $\theta$ is well-defined bijective linear map and also $\theta([(x_{1},\overline{v_1}),(x_{2},\overline{v_2})]) 
=[\theta(x_{1},\overline{v_1}),\theta(x_{2},\overline{v_2})].$ Also 
\[\varphi\theta (x,\overline{v})=\varphi(x+f(\overline{v}))=\varphi(x+u);\]
and 
\[\theta \psi (x,\overline{v})=\theta(\varphi(x),\tilde{\varphi}(\overline{v}))=\varphi(x)+f(\varphi(u)+Z(V))=\varphi(x)+\varphi(u),\]

where $v=u+z,~u \in U, z \in Z(V)$. Thus $\theta$ is the required isomorphism.
\end{proof}

If $(V,\varphi_{1})$ and $(W,\varphi_{2})$ are two isoclinic stem Hom-Lie algebras, then the next Lemma gives the connection between $(W,\varphi_{2})$ and a new Hom-Lie algebra constructed from the factor set of $(V,\varphi_{1})$.

 \begin{lemma}
Let $(V,\varphi_{1})$ be a stem Hom-Lie algebra in an isoclinism family of regular Hom-Lie algebras $\mathcal{C}$. Then for any stem Hom-Lie algebra $(W,\varphi_{2})$ of $\mathcal{C}$, there exists a factor set $r$ over $(V,\varphi_{1})$ such that $W \cong (Z(V),V/Z(V),r)$. 
\end{lemma}

\begin{proof}
Let the pair $(\alpha,\beta)$ be isoclinism of regular Hom-Lie algebras $(V,\varphi_{1})$ and $(W,\varphi_{2})$, then by Lemma 2.6 $\theta(Z(V))=Z(W)$. By virtue of Lemma 3.3, there exists a factor set $s$ such that $W\cong (Z(W),W/Z(W),s)$. If we define 
$ r:V/Z(V)\times V/Z(V) \longrightarrow Z(V) $ by
 \[~~~~~~~~r(\overline{v_1},\overline{v_2})=\theta^{-1}(s(\alpha(\overline{v_1}),\alpha(\overline{v_2})))~~~~\forall ~\overline{v_1},~\overline{v_2} \in V/Z(V), \]
 then $r$ is a skew-bilinear map. Also, since $\alpha$ is an isomorphism i.e. $\alpha \tilde{\varphi_1} = \tilde{\varphi_2} \alpha $. Thus
 \begin{equation*} 
\begin{split}
& r([\overline{v_1},\overline{v_2}],\tilde{\varphi_1}(\overline{v_3}))\\
&= \theta^{-1}(s(\alpha([\overline{v_1},\overline{v_2}]),\alpha \tilde{\varphi_1}(\overline{v_3})))\\
&= \theta^{-1}(s([\alpha(\overline{v_1}),\alpha(\overline{v_2})]), \tilde{\varphi_2}\alpha(\overline{v_3})))\\
&= -\theta^{-1}(s([\alpha(\overline{v_2}),\alpha(\overline{v_3})]), \tilde{\varphi_2}\alpha(\overline{v_1})))-\theta^{-1}(s([\alpha(\overline{v_3}),\alpha(\overline{v_1})]), \tilde{\varphi_2}\alpha(\overline{v_2})))\\
&=-r([\overline{v_2},\overline{v_3}],\tilde{\varphi_1}(\overline{v_1}))- r([\overline{v_3},\overline{v_1}],\tilde{\varphi_1}(\overline{v_2})).
\end{split}
 \end{equation*}

\noindent Thus $(Z(V),V/Z(V),r)$ is a factor set. Let us denote $R=(Z(V),V/Z(V),r)$ and $S=(Z(W),W/Z(W),s)$. Then by Lemma 3.2 $(R,\psi_{1})$ and $(S,\psi_{2})$ are Hom-Lie algebras. Let us define the map $\eta$, from the map $\alpha$ and $\beta$, i.e.
$\eta:(Z(V),V/Z(V),r)\longrightarrow (Z(W),W/Z(W),s)$  by
 \[ \eta(x,\overline{v})=(\beta(x),\alpha(\overline{v})),\]
which is clearly a bijective linear map with   $\eta([(x,\overline{v_1}),(y,\overline{v_2})]) 
=[\eta(x,\overline{v_1}),\eta(y,\overline{v_2})]$. Also
\[\eta \psi_{1} (x,\overline{v})=\eta(\varphi_{1}(x),\tilde{\varphi_1}(\overline{v}))=(\beta\varphi_{1}(x),\alpha\tilde{\varphi_1}(\overline{v})),\]
and 
\[ \psi_{2} \eta (x,\overline{v})=\psi_{2}(\beta(x),\alpha(\overline{v}))= (\varphi_{2}\beta(x),\tilde{\varphi_2}\alpha(\overline{v})) .\]

Since $\alpha$ and $\beta$ are isomorphism, we have $\psi_{1}\eta =\eta \psi_{2}$.
Hence, $\eta$ is an isomorphism and $W\cong (Z(V),V/Z(V),r).$ 
\end{proof}

\begin{lemma}\label{Lem8}
Let $(V,\varphi)$ be a regular Hom-Lie algebra, $r$ and $s$ be two multiplicative factor sets over $(V,\varphi)$. Assume that
\[R=(Z(V),L/Z(V),r),~~~~~~~~Z_{R}=\{(x,0) \in R:~x \in Z(V)\} \cong Z(V),\]
\[S=(Z(V),V/Z(V),s),~~~~~~~~Z_{S}=\{(x,0) \in S:~x \in Z(V)\} \cong Z(V).\]
Let $\eta$ is an isomorphism from $R$ to $S$ satisfying $\eta(Z_{R})=Z_{S}$, then the restriction of $\eta$ on $V/Z(V)$ and $Z(V)$ define the automorphisms $\mu \in Aut(V/Z(V))$ and $\nu \in Aut(Z(V))$, respectively.
 \end{lemma}
 
\begin{proof}
By Lemma 3.2 $(R,\psi)$ and $(S,\psi)$ are regular Hom-Lie algebras. Thus we can define the quotient Hom-Lie algebra, as $Z(R)=Z_{R}$ and $Z(S)=Z_{S}$. Since $\eta$ is an isomorphism and $\eta(Z_{R})=Z_{S}$, then  map $\overline{\eta}:(R/Z_{R},\psi_{1})\longrightarrow (S/Z_{S},\psi_{2})$ given by $\overline{\eta}((x,\overline{v})+Z_{R})=\eta(x,\overline{v})+Z_{S}$ is an isomorphism, where $\psi_{1}:R/Z_{R} \rightarrow R/Z_{R}$ and $\psi_{2}:S/Z_{S} \rightarrow S/Z_{S}$ are linear map defined as $\psi_{1}((x,\overline{v})+Z_{R})=\psi(x,\overline{v})+Z_{R}$ and $\psi_{2}((x,\overline{v})+Z_{S})=\psi(x,\overline{v})+Z_{S}$ respectively. Define 
$\mu$ such that the following diagram is commutative;
\begin{center}
  \begin{tikzpicture}[>=latex]
\node (x) at (0,0) {\(V/Z(V) \)};
\node (z) at (0,-2) {\(R/Z_{R}\)};
\node (y) at (3,0) {\(V/Z(V)\)};
\node (w) at (3,-2) {\(S/Z_{S}\)};
\draw[->] (x) -- (y) node[midway,above] {$\mu$};
\draw[->] (x) -- (z) node[midway,left] {$\lambda_{1}$};
\draw[->] (z) -- (w) node[midway,below] {$\overline{\eta}$};
\draw[->] (y) -- (w) node[midway,right] {$\lambda_{2}$};
\end{tikzpicture}\\
\end{center}
 where $\lambda_{1}$ and $\lambda_{2}$ are projection map, i.e. $\lambda_{1}( \overline{v})= (0, \overline{v})+Z_{R}$ and $\lambda_{2}(\overline{v})=(0, \overline{v})+Z_{S}$. Thus $\eta(0,\overline{v})+Z_{S}=(0,\mu(\overline{v}))+Z_{S}$ for all $\overline{v} \in V/Z(V)$. 
 
\noindent Now
\[(0,\mu \tilde{\varphi}(\overline{v}))+Z_{S}=\eta(0,\varphi(v)+Z(V)) +Z_{S}=\eta\psi(0,\overline{v})+Z_{S}.   \] 
 Also on the another side
 
 \[(0, \tilde{\varphi}\mu(\overline{v}))+Z_{S}=\psi(0,\mu(\overline{v}))+Z_{S}=\psi(\eta(0,\overline{v}) +x)+Z_{S}=\psi\eta(0,\overline{v})+\psi(x)+Z_{S}=\psi\eta(0,\overline{v})+Z_{S},
    \]
 where $x \in Z_{S}$. Since $\eta$ is an automorphism i.e. $\eta\psi=\psi\eta$ and $\lambda_{2}$ is injective, we have
 
 \begin{equation*} 
\begin{split}
  &(0,\mu \tilde{\varphi}(\overline{v}))+Z_{S} = (0, \tilde{\varphi}\mu(\overline{v}))+Z_{S}\\
 & \Rightarrow \lambda_{2}(\mu \tilde{\varphi} (\overline{v}))= \lambda_{2}(\tilde{\varphi}\mu (\overline{v}))\\
& \Rightarrow \mu \tilde{\varphi}(\overline{v})= \tilde{\varphi} \mu(\overline{v})
 \end{split}
\end{equation*}
Certainly $\mu$ is bijective and $\mu([\overline{v_{1}},\overline{v_2}])=[\mu(\overline{v_1}),\mu(\overline{v_2})]$. Thus $\mu$ is an automorphism.
 Consider the map $\tilde{\eta}: Z_{R} \longrightarrow Z_{S}$ is defined as $\tilde{\eta}(x,0)=\eta(x,0)$ for all $x \in Z(V)$, is an isomorphism. Define $\nu$ in such a way that the following diagram is commutative;
\begin{center}
\begin{tikzpicture}[>=latex]
\node (x) at (0,0) {\( Z(V) \)};
\node (z) at (0,-2) {\(Z_{R}\)};
\node (y) at (2,0) {\(Z(V)\)};
\node (w) at (2,-2) {\(Z_{S}\)};
\draw[->] (x) -- (y) node[midway,above] {$\nu$};
\draw[->] (x) -- (z) node[midway,left] {$\overline{\lambda_1}$};
\draw[->] (z) -- (w) node[midway,below] {$\tilde{\eta}$};
\draw[->] (y) -- (w) node[midway,right] {$\overline{\lambda_2}$};
\end{tikzpicture}\\
\end{center}
 
\noindent  where $\overline{\lambda_1}$ and $\overline{\lambda_2}$ are projection map and  $\eta(x,0)=(\nu(x),0)$ for all $x \in Z(V)$. One can easily verify that $\nu$ is an automorphism.
\end{proof}

\begin{lemma}\label{lem9}
Let $(V,\varphi)$ be a regular Hom-Lie algebra and $(R,\psi), (S,\psi), Z_{R}$ and $Z_{S}$ be as in the previous Lemma.
\begin{enumerate}
\item Consider $\eta : R \longrightarrow S$ is a Hom-Lie algebra isomorphism such that $\eta(Z_{R})= Z_{S}$. Let $\mu \in Aut(V/Z(V))$ and $\nu \in Aut(Z(V))$ be the automorphisms induced by $\eta$. Then there exists a linear map, $\gamma: V/Z(V) \longrightarrow Z(V)$ such that,\[\nu(r(\overline{v_1}, \overline{v_2})+ \gamma[\overline{v_1}, \overline{v_2}])=s(\mu(\overline{v_1}), \mu(\overline{v_2})).\]
\item If $\mu \in Aut(V/Z(V))$, $\nu \in Aut(Z(V))$ and $\delta: V/Z(V) \longrightarrow Z(V)$ is a linear map  such that $$\nu(r(\overline{v_1}, \overline{v_2})+ \delta[\overline{v_1}, \overline{v_2}])=s(\mu(\overline{v_1}), \mu(\overline{v_2})) ,~~~~~~~ \delta \tilde{\varphi}=\varphi \delta.$$ Then there exists an isomorphism $\eta:R \longrightarrow S$ which is induced by $\mu$ and $\nu$ satisfying $\eta(Z_{R})=Z_{S}$.
\end{enumerate}
\end{lemma}

\begin{proof}
The proof of the first assertion is similar to [10, Lemma 3.6]. In the second assertion, we will proof the only commutative property of $\eta$, where 
 $\eta: R \longrightarrow S $  is given by
\[\eta(x, \overline{v}) = ( \nu(x)+ \delta(\overline{v}), \mu(\overline{v})).\]

\noindent  In one hand

\[\eta \psi (x, \overline{v})=\eta(\varphi(x),\tilde{\varphi}(\overline{v}))=(\nu \varphi(x)+\delta \tilde{\varphi}(\overline{v}),\mu \tilde{\varphi}(\overline{v})).\]
 
\noindent On the other hand
\[\psi \eta (x, \overline{v})= \psi ( \nu(x)+ \delta(\overline{v}), \mu(\overline{v}))=( \varphi \nu(x)+ \varphi \delta(\overline{v}), \tilde{\varphi} \mu(\overline{v})). \]
Since $\mu$ and $\nu$ are automorphisms and $\delta \tilde{\varphi}=\varphi \delta $, then $\eta \psi =\psi \eta.$
\end{proof}

\begin{theorem}\label{Th1}
Let $(V,\varphi_{1})$ and $(W,\varphi_{2})$ be two finite dimensional stem Hom-Lie algebras. Then $V \sim W$ if and only if $V \cong W$.
\end{theorem}
\begin{proof}
 Suppose $V \sim W$. By Lemma 3.3 and Lemma 3.4 we have, $V \cong (Z(V), \frac{V}{Z(V)}, r)=R$ and also  $W \cong (Z(V), \frac{V}{Z(V)}, s)=S$. Since $(V,\varphi_{1})$ and $(W,\varphi_{2})$ are regular, then Lemma 1.4 tells us that $(R,\psi_1)$ and $(S,\psi_2)$ are also regular. Now let $(\alpha, \beta)$ be the isoclinism between the regular Hom-Lie algebras $(R,\psi_1)$ and $(S,\psi_2)$. Certainly $Z_{R}=Z(R)$ and $Z_{S}=Z(S)$.

 Let the map $\mu \in Aut(V/Z(V))$, be defined by $\alpha((0, \overline{v})+Z_{R})= (0, \mu(\overline{v}))+Z_{S}$, for all $\overline{v} \in V/Z(V)$. Let us consider the following commutative diagram;
\begin{center}
 \begin{tikzpicture}[>=latex]
\node (A_{1}) at (0,0) {\(\frac{V}{Z(V)}\times \frac{V}{Z(V)} \)};
\node (A_{2}) at (4,0) {\(\frac{R}{Z_R}\times  \frac{R}{Z_R}\)};
\node (A_{3}) at (8,0) {\(R'\)};
\node (B_{1}) at (0,-2) {\(\frac{V}{Z(V)}\times \frac{V}{Z(V)}\)};
\node (B_{2}) at (4,-2) {\(\frac{S}{Z_S}\times \frac{S}{Z_S} \)};
\node (B_{3}) at (8,-2) {\(S'\)};

\draw[->] (A_{1}) -- (A_{2}) node[midway,above] {$\rho$};
\draw[->] (A_{2}) -- (A_{3}) node[midway,above] {$\theta$};
\draw[->] (B_{1}) -- (B_{2}) node[midway,below] {$\sigma$};
\draw[->] (B_{2}) -- (B_{3}) node[midway,below] {$\xi$};

\draw[->] (A_{1}) -- (B_{1}) node[midway,right] {$\mu^2$};
\draw[->] (A_{2}) -- (B_{2}) node[midway,right] {$\alpha^2$};
\draw[->] (A_{3}) -- (B_{3}) node[midway,right] {$\beta$};

\end{tikzpicture}
\end{center}

 in which
  \begin{equation*}
  \begin{split}
 \rho(\overline{v_1}, \overline{v_2}) = ((0, \overline{v_1})+Z_{R}, (0, \overline{v_2})+Z_{R}),\\
 \sigma(\overline{v_1}, \overline{v_2}) =  ((0, \overline{v_1})+Z_{S}, (0, \overline{v_2})+ Z_{S}),\\
\xi((x_1, \overline{v_1})+Z_{S}, (x_2,\overline{v_2})+Z_{S})= [(x_1,\overline{v_1}),(x_2, \overline{v_2})]=(s(\overline{v_1}, \overline{v_2}), [\overline{v_1}, \overline{v_2}]),\\
\theta((x_1, \overline{v_1})+Z_{R}, (x_2,\overline{v_2})+Z_{R})=(r(\overline{v_1}, \overline{v_2}), [\overline{v_1}, \overline{v_2}]).
 \end{split}
 \end{equation*}

  Again let $\nu \in Aut(Z(L))$ be defined by $\beta(x, 0)=(\nu(x),0)$, for all $x \in Z(L)$. Now for $\overline{v_1}, \overline{v_2} \in V/Z(V)$, consider
$$
\beta \theta((0, \overline{v_1})+Z_{R}, (0, \overline{v_2})+Z_{R}) = \beta[(0,\overline{v_1}),(0, \overline{v_2})]
$$
 and further \begin{align*}
 \xi \alpha((0, \overline{v_1})+Z_{R}, (0, \overline{v_2})+Z_{R})&=\xi((0, \mu(\overline{v_1}))+Z_{S}, (0, \mu(\overline{v_2}))+Z_{S} )\\
 &= [(0, \mu(\overline{v_1})), (0, \mu(\overline{v_2}))]\\
 &= (s((\mu(\overline{v_1}), \mu(\overline{v_2})), [\mu(\overline{v_1}), \mu(\overline{v_2})]).
 \end{align*}
 Hence, we have $\beta[(0,\overline{v_1}),(0, \overline{v_2})]
=(s((\mu(\overline{v_1}), \mu(\overline{v_2})), [\mu(\overline{v_1}), \mu(\overline{v_2})])$.
 Let us define the map $\delta: V'/Z(V) \longrightarrow Z(V)$ such that
$$ \beta(0, [\overline{v_1}, \overline{v_2}])=(\delta( [\overline{v_1}, \overline{v_2}]),t)$$
 
\noindent  where $t \in V/Z(V)$ and thus, we get
 \[\nu(r( \overline{a}, \overline{b})+\delta( [\overline{a}, \overline{b}])=s((\mu(\overline{a}), \mu(\overline{b})).\]
 To apply Lemma 3.6, we may extend $\delta$ to $V/Z(V)$ by defining $0$ on the complement of $V^{'}/Z(V)$ in $V/Z(V)$. Now we will show $\delta \tilde{\phi}=\phi \delta .$ We have
\[(\delta \tilde{\varphi}[\overline{v_1}, \overline{v_2}],t_{1})=\beta(0,\tilde{\varphi}[\overline{v_1}, \overline{v_2}])=\beta\psi(0,[\overline{v_1}, \overline{v_2}]).\] 
 
\noindent Further $\tilde{\varphi}$ is surjective, i.e. $\tilde{\varphi} (t)=t_2$, for any $t_2 \in V/Z(V)$. Therefore
 \[(\varphi \delta [\overline{v_1}, \overline{v_2}],t_2)= (\varphi \delta [\overline{v_1}, \overline{v_2}],\tilde{\varphi}(t))=\psi(\delta [\overline{v_1}, \overline{v_2}],t )=\psi\beta (0,[\overline{v_1}, \overline{v_2}]) .\]
 As a result $\varphi \delta(\overline{v})=\delta \tilde{\varphi}(\overline{v}) $ for all $v \in V'$. Suppose $V=V'\oplus U$, then define $\delta$ to be zero in $U$. Then $\varphi\delta (\overline{u})=0$ for all $u \in U$. Since $\varphi$ is injective, $\varphi(u) \in U$ for $u \in U$, we have 
 \[ \delta \tilde{\varphi}(\overline{u})=\delta (\varphi(u)+Z(V))=0.\]
 Thus $\delta \tilde{\phi}=\phi \delta $ and now we apply Lemma 3.6 to get our results.
\end{proof}

Now with the help of Lemmas, Theorem proved earlier up to this part of the paper we can easily prove Theorem 3.8 and Theorem 3.9. The proofs of these two Theorems are similar to the proofs done earlier in [10].

\begin{theorem}\label{Th2}
Let $\mathcal{C}$ be an isoclinism family of finite dimensional regular Hom-Lie algebras. Then any $V \in \mathcal{C}$ can be expressed as $V= T \oplus A$ where $T$ is a stem Hom-Lie algebra and $A$ is some finite dimensional abelian Hom-Lie algebra.
\end{theorem}

\begin{theorem}\label{Th3}
If $(V,\varphi_1)$ and $(W,\varphi_2)$ be two regular Hom-Lie algebras with same dimension. Then $V \sim W$ if and only if $V \cong W$.
\end{theorem}

 Below is an example which shows that, two isoclinic regular Hom-Lie algebras of different dimensions may not be isomorphic.

\begin{example}
Let $(V, \varphi_1)$ be a 3 dimensional regular Hom-Lie algebra with the basis $\{v_{1},v_{2},v_{3}\}$ and the skew-bilinear map is defined by;
\[[v_{1},v_{2}]=v_{2},~~[v_{1},v_{3}]=v_{3},\]
and all other commutator relations are zero. The linear map $\varphi_1$ is defined as
\[\varphi_{1}(v_1)=v_1,~\varphi_{1}(v_2)=v_3,~\varphi_{1}(v_3)=v_2 .\]
 Then $V'=<v_{2},v_{3}>$ and $Z(V)=0$ and hence, $V/Z(V) \cong V$.\\

Now, let $(W, \varphi_2)$ be a 4 dimensional regular Hom-Lie algebra with the basis $\{w_{1},w_{2},w_{3},w_{4}\}$ and  the commutator relations are defined by;
\[[w_{1},w_{2}]=w_{2},~~[w_{1},w_{3}]=w_{3},\]
and all other commutator relations are zero. Also the linear map is given by
\[\varphi_{2}(w_1)=w_1,~\varphi_{2}(w_2)=w_3,~\varphi_{2}(w_3)=w_2,~\varphi_{2}(w_4)=w_4 .\]
 Then $W'=<w_{2},w_{3}>$ and $Z(W)=\{w_{4}\}$ and hence, $W/Z(W)=\{ \overline{w_{1}},\overline{w_{2}},\overline{w_{3}}\}$, where $\overline{w_{i}}=w_{i}+Z(W)$, for $i=1,2,3$.\\

It is easy to verify that $V' \cong W'$ and $V/Z(V) \cong W/Z(W)$. Hence, one can deduce $V\sim W$ while $dim(V)\neq dim(W)$.
\end{example}

\end{document}